\documentclass{article}

\usepackage{amsfonts}
\usepackage{amsmath}
\usepackage{theorem}
\usepackage{array}
\usepackage[a4paper]{geometry}
\usepackage{amssymb}
\usepackage{graphics}
\usepackage[usenames]{color}
\usepackage[all]{xy}
\usepackage{graphicx}
\usepackage{boxedminipage}
\usepackage{url}

\newtheorem{theorem}{Theorem}[section]

\newtheorem{definition}[theorem]{Definition}
\newtheorem{remark}[theorem]{Remark}

\newtheorem{lemma}[theorem]{Lemma}

\newtheorem{proposition}[theorem]{Proposition}

\newenvironment{proof}[1][Proof]{\textbf{#1.} }{\ \rule{0.5em}{0.5em}}

\newcommand{\im}{\mathrm{im}}

\newcommand{\fqn}{\mathbb{F}_{q^{2n}}}
\newcommand{\fqt}{\mathbb{F}_{q^2}}

\newcommand{\aut}{\mathrm{Aut}}

\begin{document}

\title{A complete characterization of Galois subfields of the generalized Giulietti--Korchm\'aros function field}
\author{Nurdag\"{u}l Anbar, Alp Bassa and Peter Beelen}
\date{}
\maketitle

\begin{abstract}
We give a complete characterization of all Galois subfields of the generalized Giulietti--Korchm\'aros function fields $\mathcal C_n / \fqn$ for $n\ge 5$. Calculating the genera of the corresponding fixed fields, we find new additions to the list of known genera of maximal function fields.
\end{abstract}

\noindent
AMS: 11G20, 14H25, 14H37, 14G05

\vspace{1ex}

\section{Introduction}
Let $F$ be a function field of genus $g(F)$ over the finite field $\mathbb F_{\ell}$ with ${\ell}$ elements. The Hasse--Weil theorem gives the following upper bound for the number of rational places $N(F)$ of $F$:
$$N(F)\leq \ell+1+2g(F)\sqrt{\ell} \ .$$
Function fields attaining this bound are called maximal, and have played a central role in the theory of function fields over finite fields (or equivalently curves over finite fields).

An  important example of maximal function fields is the Hermitian function field $\mathcal H$ over the finite field $\mathbb F_{q^2}$. It is given by $\mathcal H=\mathbb F_{q^2}(x,y)$ with
$$x^q+x=y^{q+1} \ .$$
It has genus $q(q-1)/2$ (in fact the largest possible genus for a maximal function field over $\mathbb F_{q^2}$) and a large automorphism group $A \cong {\mathrm PGU}(3,q)$. By a theorem of Serre (see \cite{Lach}) a subfield of a maximal function field has to be maximal. Studying subfields of the Hermitian function field leads to many new examples of maximal function fields. One way to construct such subfields is by taking fixed fields of subgroups of $A$ (see among others \cite{AQ,bassa,GSX} and the references in \cite{HKT}). Since all maximal subgroups of $A$ are known, interest has been diverted into studying subgroups of the various maximal subgroups. The maximal subgroup $A(P_{\infty})$ fixing the unique pole $P_{\infty}$ of $x$, together with an involution generates the whole automorphism group $A$. A complete characterization of all subgroups of $A(P_\infty)$ and the genera of the corresponding fixed fields have been given in \cite{bassa}.

For a long time, all known maximal function fields were subfields of the Hermitian function field. This led to the question whether any maximal function fields could be embedded as subfields in the Hermitian function field.
Giulietti and Korchm\'aros \cite{GK} introduced a new family of maximal function fields (GK function fields) over finite fields $\mathbb F_{q^6}$, which are not subfields of the Hermitian function field over the corresponding field for $q>2$. The GK function field is given by $\mathcal C=\mathbb F_{q^6}(x,y,z)$ with
$$x^q+x=y^{q+1}\ \makebox{ and } \ z^{(q^3+1)/(q+1)}=y \sum_{i=0}^q(-1)^{i+1}x^{i(q-1)} \ .$$
The GK function field was generalized in \cite{GGS} to a family of maximal function fields over finite fields $\mathbb F_{q^{2n}}$ with $n$ odd. The generalized GK (GGK) function field, also known as the Garcia--G\"{u}neri--Stichtenoth function field, is given by $\mathcal C_n=\mathbb F_{q^{2n}}(x,y,z)$ with
$$x^q+x=y^{q+1}\ \makebox{ and } \ z^{(q^n+1)/(q+1)}=y^{q^2}-y \ .$$
One recovers the Hermitian function field $\mathcal H$ for $n=1$, and the GK function field $\mathcal C$  for $n=3$, since
$$\sum_{i=0}^q(-1)^{i+1}x^{i(q-1)}=-1+x^{q-1}\sum_{j=0}^{q-1}(-1)^jx^{j(q-1)}=-1+x^{q-1}(x^{q-1}+1)^{q-1}=-1+\left(y^{q+1}\right)^{q-1}.$$
 Note that the GGK function field contains a constant field extension of the Hermitian function field $\mathcal H$ over $\mathbb F_{q^2}$ as a subfield. The GGK function fields are not Galois subfields of the Hermitian function fields (\cite{DM,GMZ}).

The automorphism groups of the GGK function fields were determined in \cite{auto1, auto2}. In particular, it was shown that for $n>3$ any automorphism of $C_n$ restricts to an automorphism of $\mathcal H$ fixing $P_\infty$.  We use this together with the characterization in \cite{bassa} to characterize all subgroups of the automorphism group of $\mathcal C_n$ for $n >3$. For $n=3$ the automorphisms of $\mathcal C_n$ do not restrict necessarily to an automorphism of $\mathcal H$ fixing $P_\infty$, but we obtain a characterization of a large class of subgroups.
Adapting an approach from \cite{GSX} in a similar way as in \cite{DO}, we obtain an explicit expression for the genus of the fixed field of these characterized subgroups. This leads to new additions to the list of known genera of maximal function fields.

\section{Results about the Hermitian function field}

We denote by $\mathcal H$ the Hermitian function field over $\fqt$. It can be given as $\mathcal H=\fqt(x,y)$ with $x^q+x=y^{q+1}$. The functions $x$ and $y$ have exactly one pole at a place, which we denote by $P_{\infty}$. As is well known $\mathcal H$ has a large automorphism group $A$ isomorphic to $\mathrm{PGU}(3,q)$. We denote by $A(P_\infty)$ the stabilizer of $P_\infty$. For some $a \in \fqt^*, b,c \in \fqt$, an automorphism $\sigma \in A(P_\infty)$ can be described by the equations
\begin{equation*}
  \sigma(x)=a^{q+1}x+ab^qy+c \quad \text{and}\quad \sigma(y)=ay+b \quad \text{with } c^q+c=b^{q+1} \ .
\end{equation*}
Instead of $\sigma$, we write $[a,b,c]$ for this automorphism. Then $A(P_\infty)$ is given by
\begin{equation}\label{eq:autH}
A(P_\infty)=\{[a,b,c] \, | \, a \in \fqt^*, b,c \in \fqt \ \makebox{ where } \ c^q+c=b^{q+1}\} \ .
\end{equation}
The group law is given by
\begin{equation*}
[a',b',c']\circ[a,b,c]=[a'a,ab'+b,a^{q+1}c'+ab^{q}b'+c]
\end{equation*}

Following \cite{bassa}, for a given subgroup $H \leq A(P_\infty)$, we define the map $\phi: H \rightarrow \fqt^*$ such that $[a,b,c] \mapsto a$.
We denote by $U_H$ the kernel $\ker \phi$ and by $H_1$ the image $\im \phi$ of $\phi$. Then $U_H=\{[1,b,c] \in H\}$ is the unique $p$-Sylow subgroup of $H$. Further, define $\psi: U_H \rightarrow \fqt$ such that $[1,b,c] \mapsto b$. Let $H_2=\im \psi$ and $H_3=\ker \psi.$ For convenience, we write $h_1:=\#H_1$ and $h_w:=\#U_H$. Note that we have $\#H=h_1h_w$, and $h_w=\#H_2 \# H_3$.

In \cite{bassa} it was described precisely which triples $(H_1,H_2,H_3)$ up to conjugation can arise as $H$ varies over all subgroups of $A(P_\infty)$. It will be convenient for a subset $S \subset \fqt$ and integer $i \ge 0$ to denote by $\mathbb{F}_p(S^i)$ the field obtained from $\mathbb{F}_p$ by adjoining all $i$th powers of elements from $S$. Then, we have the following theorem.
\begin{theorem}[Theorem 3.6 \cite{bassa}]\label{thm:classifyH}
For each subgroup $H \leq A(P_\infty)$ up to conjugation we have the following:
\begin{enumerate}
\item[(i)] $H_1$ is a cyclic subgroup of $\fqt^*$.
\item[(ii)] $H_2 \subset \fqt$ is a vector space over $\mathbb{F}_p(H_1)$.
\item[(iii)] $H_3 \subset \{c \in \fqt \mid c^q+c=0\}$ is a vector space over $\mathbb{F}_p(H_1^{q+1})$ containing $W$, where $W=\{b_1b_2^q-b_2b_1^q \, | \, b_1,b_2 \in H_2\}$ if $p$ is odd and $W=\{b^{q+1} \, | \, b \in H_2\}$ if $p=2$.
\end{enumerate}
Conversely, for any $H_1$, $H_2$, $H_3$ satisfying (i), (ii) and (iii) there exists a subgroup $H \leq A(P_\infty)$ giving rise to this triple.
\end{theorem}
Given the triple $(H_1,H_2,H_3)$ the genus of the fixed field of $H$ can be determined in terms of $h_1,\#H_2$ and $\#H_3$, see \cite{GSX}.
Theorem \ref{thm:classifyH} characterizes all possible subgroups $H \leq A(P_\infty)$ up to conjugation, but to make this result algorithmic one needs to be able to compute all possibilities for $(h_1, \#H_2,\#H_3)$ in a fast way. In \cite{bassa} for a given odd $q$ all possible values of $h_1,\#H_2$ and $\#H_3$ have been determined explicitly, while for even $q$ many (but possibly not all) possibilities were listed. In this algorithmic sense, the case $q=2$ is not completely settled.

\section{Subgroups of the automorphism group of the GGK function field.}

For any odd $n \ge 1$, the generalized Giulietti--Korchm\'aros (GGK) function field $\mathcal C_n/\fqn$ was introduced in \cite{GGS}. It is defined as $\mathcal C_n=\fqn(x,y,z)$ satisfying the equations
$$x^q+x=y^{q+1} \quad \text{and} \quad z^m=y^{q^2}-y \quad \text{with  } m=\frac{q^n+1}{q+1} \ .$$
The function field $\mathcal C_n$ is maximal over $\fqn$ and has genus $g(\mathcal C_n)=(q-1)(q^{n+1}+q^n-q^2)/2$. The function $z$ has a unique pole in $\mathcal C_n$, which we denote by $Q_\infty$. As mentioned previously $\mathcal C_1=\mathcal H$, the Hermitian function field, and $\mathcal C_3=\mathcal C$, the GK function field.

The automorphism group $B:=\aut(\mathcal C_n)$ of $\mathcal C_n$ has been determined in \cite{auto1,auto2}. The stabilizer of $Q_\infty$, which we will denote by $B(Q_\infty)$, is in most cases equal to the entire automorphism group. More precisely, we have
\begin{equation}\label{eq:index}
[B:B(Q_\infty)]= \left\{ \begin{array}{rl} q^3+1 & \makebox{ if $n \le 3$}\\ 1 & \makebox{ if $n \ge 5$}\end{array}\right.
\end{equation}
Note that for $n=1$, we simply have $Q_\infty=P_\infty$ and $B(Q_\infty)=A(P_\infty)$. For general odd $n$, the group $B(Q_\infty)$ can be described explicitly. All automorphisms $\sigma \in B(Q_\infty)$ can be obtained in the following way: for $a \in \fqt^*, b,c \in \fqt$ and $d \in \fqn$, define
\begin{equation*}
  \sigma(x)=a^{q+1}x+ab^qy+c , \quad \sigma(y)=ay+b \quad \text{and}\quad \sigma(z)=dz \quad \text{with } c^q+c=b^{q+1} \ , \;  d^m=a \ .
\end{equation*}
It will be convenient to write $[a,b,c,d]$ for $\sigma$. Then we have
\begin{equation}\label{eq:autC}
B(Q_\infty)=\{[a,b,c,d] \, | \, a \in \fqt^*, b,c \in \fqt \ \makebox{ where } \ c^q+c=b^{q+1} \ \makebox{ and } \ d^m=a\} \ .
\end{equation}
Note that for any $a \in \fqt^*$, the equation $d^m=a$ has $m$ distinct solutions in $\fqn$.
We obtain that $\# B(Q_\infty)=mq^3(q^2-1)$. The group law is easily seen to be
\begin{equation}\label{eq:grouplaw}
[a',b',c',d']\circ[a,b,c,d]=[a'a,ab'+b,a^{q+1}c'+ab^{q}b'+c,d'd]
\end{equation}

It is clear from Equations \eqref{eq:autH} and \eqref{eq:autC} that the map
\begin{equation*}
\pi: B(Q_\infty) \rightarrow A(P_\infty) \; \ \makebox{defined by} \ \; \pi([a,b,c,d])=[a,b,c]
\end{equation*}
is a surjective group homomorphism. Note that the GGK function field over $\fqn$ contains the Hermitian function field over $\fqt$ as a subfield. Then the map $\pi$ corresponds to restricting automorphisms of the GGK function field to this subfield.

For any $[a,b,c,d] \in B(Q_\infty)$, we have $$d^{(q^n+1)(q-1)}=d^{m(q^2-1)}=a^{q^2-1}=1\ .$$ Hence $d \in \mu$, where $\mu \leq \fqn^*$ is the multiplicative group of $(q^n+1)(q-1)$th roots of unity.
\begin{definition}\label{def:G0G1}
Write $\pi_d: B(Q_\infty) \rightarrow \mu$ for the map given by $\pi_d([a,b,c,d]):=d.$ Then for any $G \subset B(Q_\infty)$, we define $G_0:=\pi_d(G) \subset \mu$ and $g_0:=\#G_0$.
Similarly, write $\pi_a: B(Q_\infty) \rightarrow \fqt^*$ for the map given by $\pi_a([a,b,c,d]):=a$. Then we define $G_1:=\pi_a(G)$ and $g_1:=\#G_1$.
\end{definition}

Note that $\pi_d$ and $\pi_a$ are group homomorphisms, and that $\pi_a=e_m \circ \pi_d$, where $e_m$ is the $m$th power map. We denote by $\pi|_G$, $\pi_a|_G$ and $\pi_d|_G$ the restriction to $G$ of the maps $\pi$, $\pi_a$ and $\pi_d$. The map $\pi$ can be used to associate objects to a subgroup $G \leq B(Q_\infty)$. Indeed, writing $H=\pi(G) \in A(P_\infty)$, we can construct the triple $(H_1,H_2,H_3)$ from the previous section. To stress the dependency on $G$, we will write $(G_1,G_2,G_3)$ instead. At first sight this gives a problem, since the notation $G_1$ was already used in Definition \ref{def:G0G1}. However, a direct computation shows that
$$e_m\circ \pi_d|_G=\pi_a|_G=\phi\circ\pi|_G,$$ with $\phi: \pi(G) \rightarrow \fqt^*$ given by $\phi([a,b,c])=a$, just as in the previous section. Therefore the two definitions for $G_1$ actually give rise to the same group. Analogously to the Hermitian function field case, we also define $g_w:=\#\ker(\pi_d|_G)$.

The situation is as depicted in the following picture.
\begin{displaymath}
\xymatrix{
&{\rm id} \ar@{->}[d]\\
&G_3 \ar@{->}[d]&G\ar@{->}[d]_{\pi|_G} \ar@{->}[dr]^{\pi_a|_G}\ar@{->}[r]^{\pi_d|_G}&G_0 \ar@{->}[d]^{e_m} \ar@{->}[r] &{\rm id}\\
{\rm id} \ar@{->}[r]& U_{\pi(G)}\ar@{->}[r] \ar@{->}[d]_{\psi}& \pi(G) \ar@{->}[r]_{\phi} & G_1 \ar@{->}[r] & {\rm id}\\
&G_2 \ar@{->}[d]\\
&{\rm id}\\
}
\end{displaymath}
With this notation, we obtain the following.
\begin{lemma}\label{lem:cardGpiG}
Let $G \leq B(Q_\infty)$ be a subgroup. Then we have
\begin{enumerate}
\item[(i)] $\# G=g_0g_w,$
\item[(ii)]$\#\pi(G)=g_1g_w.$
\end{enumerate}
\end{lemma}
\begin{proof}
The different maps and groups are depicted in the following commutative diagram.

\begin{displaymath}
\xymatrix{
G\ar@{->}[d]_{\pi|_G}\ar@{->}[dr]^{\pi_a|_G}\ar@{->}[r]^{\pi_d|_G}&\mu\ar@{->}[d]^{e_m} \\
\pi(G)\ar@{->}[r]_{\phi}& \mathbb{F}_{q^2}^{\star}
}
\end{displaymath}
Considering the map $\pi_d|_G$, we deduce $$\#G=\#\im \pi_d|_G \#\ker \pi_d|_G=\#G_0 g_w=g_0g_w \ .$$ Note that $\ker \pi_d|_G=\{[1,b,c,1] \in G\}$ and $\ker \phi =\{[1,b,c] \in \pi(G)\}$. Hence $\#\ker \phi=\#\ker \pi_d|_G=g_w$. Therefore
$$\# \pi(G)=\# \phi(\pi(G)) \# \ker \phi=\# \pi_a(G) g_w=g_1g_w\ .$$
\end{proof}

Also we can deduce the following relation between $g_0$ and $g_1$.
\begin{lemma}\label{lem:g0g1}
Let $G \leq B(Q_\infty)$ be given. Then $$G_1=G_0^m \; \ \makebox{and} \ \; g_1=\frac{g_0}{\gcd(g_0,m)} \ .$$
\end{lemma}
\begin{proof}
The identity $G_1=G_0^m$ follows from the fact that $d^m=a$ for all $[a,b,c,d] \in G$, or in other words from the identity $e_m\circ \pi_d|_G=\pi_a|_G$. Since $G_0$ is a subgroup of the cyclic group $\mu$, it is itself cyclic of order $g_0$. Therefore the kernel of the $m$th power map from $G_0$ to $G_0^m$ has cardinality $\gcd(g_0,m)$. This implies that the $\# G_0^m=\#G_0 / \gcd(g_0,m)$, which proves the claim.
\end{proof}

\begin{remark} \label{rem:kernel} Using the previous two lemmas we see that
$\#\ker\pi|_G=\#G/\#\pi(G)={\rm gcd}(g_0,m)$.
\end{remark}

We now describe the subgroups of $B(Q_\infty)$ in terms of subgroups of $\mu$ and of $A(P_\infty)$.

\begin{theorem}\label{thm:GGKreducetoH}
Let $\Sigma$ be the set of pairs $(H,M)$ such that $H \le A(P_\infty)$, $M \leq \mu$ and $M^m=\phi(H)$ and let $\Xi$ for any $G \le B(Q_\infty)$ be defined as $\Xi(G)=(\pi(G),\pi_d(G))$. Then the map $\Xi$ gives a bijection between the set of subgroups of $B(Q_\infty)$ and $\Sigma$.
\end{theorem}
\begin{proof}
Note that Lemma \ref{lem:g0g1} implies that $\Xi(G) \in \Sigma$. First we show surjectivity of $\Xi$. Let $H \le A(P_\infty)$ and $M \le \mu$ be subgroups with $M^m=\phi(H)$. Then we claim that $G:=\pi^{-1}(H) \cap \pi_d^{-1}(M)$ has the property that $\Xi(G)=(H,M)$. In other words, we need to show that
\begin{equation*}
\pi(\pi^{-1}(H) \cap \pi_d^{-1}(M))=H \ \makebox{and} \ \pi_d(\pi^{-1}(H) \cap \pi_d^{-1}(M))=M.
\end{equation*}
It is clear that $\pi(\pi^{-1}(H) \cap \pi_d^{-1}(M)) \subset \pi(\pi^{-1}(H)) =H$ and $\pi_d(\pi^{-1}(H) \cap \pi_d^{-1}(M)) \subset \pi_d(\pi_d^{-1}(M)) =M$. On the other hand, using that $M^m=\phi(H)$, we see that for any $[a,b,c] \in H$ there exists $d \in M$ such that $d^m=a$. This implies $[a,b,c,d] \in \pi^{-1}(H) \cap \pi_d^{-1}(M)$, whence $[a,b,c] \in \pi(\pi^{-1}(H) \cap \pi_d^{-1}(M)).$
Similarly for any $d \in M$ there exists $[a,b,c] \in H$ such that $a=d^m$. Then $[d^m,b,c,d] \in  \pi^{-1}(H) \cap \pi_d^{-1}(M)$ and hence $d \in \pi_d(\pi^{-1}(H) \cap \pi_d^{-1}(M))$. This shows that $\Xi$ is surjective.

Now we show that $\Xi$ is injective. Let $G \leq B(Q_\infty)$ be chosen arbitrarily. By definition we have $\Xi(G)=(\pi(G),\pi_d(G))$. The proof of the surjectivity of $\Xi$ implies that the subgroup $\pi^{-1}(\pi(G)) \cap \pi_d^{-1}(\pi_d(G))$ has the same image as $G$ under $\Xi$. To show injectivity, it is enough to show that
\begin{equation}\label{eq:intersection}
G=\pi^{-1}(\pi(G)) \cap \pi_d^{-1}(\pi_d(G)).
\end{equation}
Indeed, if $\Xi(G)=\Xi(\tilde{G})$, then Equation \eqref{eq:intersection} would imply that
$$G=\pi^{-1}(\pi(G)) \cap  \pi_d^{-1}(\pi_d(G)) = \pi^{-1}(\pi(\tilde{G})) \cap \pi_d^{-1}(\pi_d(\tilde{G}))=\tilde{G},$$
where the middle equality uses $\Xi(G)=\Xi(\tilde{G})$ and the first (resp. last) equality uses Equation \eqref{eq:intersection} for the subgroup $G$ (resp. $\tilde{G}$).
Now we prove Equation \eqref{eq:intersection} itself.

It is easy to see that $G \subset \pi^{-1}(\pi(G))$ and $G \subset \pi_d^{-1}(\pi_d(G))$, implying that $G \subset \pi^{-1}(\pi(G)) \cap \pi_d^{-1}(\pi_d(G))$. What is left is to show the reverse inclusion. Let $g \in \pi^{-1}(\pi(G)) \cap \pi_d^{-1}(\pi_d(G))$ and write $g=[a,b,c,d]$. Since in particular $g \in \pi^{-1}(\pi(G))$, there exists $\tilde{d}$ such that $[a,b,c,\tilde{d}] \in G$.
Then
\begin{equation}\label{eq:help}
[a,b,c,d]\circ[a,b,c,\tilde{d}]^{-1}=[1,0,0,d\tilde{d}^{-1}].
\end{equation}
Since $G \subset \pi^{-1}(\pi(G)) \cap \pi_d^{-1}(\pi_d(G)),$ we see that $[1,0,0,d\tilde{d}^{-1}] \in \pi_d^{-1}(\pi_d(G)) \cap \pi^{-1}(\pi(G))$ and moreover Equation \eqref{eq:help} implies that
\begin{equation}\label{eq:help2}
g \in G \ \makebox{ if and only if} \ [1,0,0,d\tilde{d}^{-1}] \in G
\end{equation}
Since $[1,0,0,d\tilde{d}^{-1}] \in \pi^{-1}(\pi(G)) \cap \pi_d^{-1}(\pi_d(G))$, in particular $[1,0,0,d\tilde{d}^{-1}] \in \pi_d^{-1}(\pi_d(G))$, implying that there exist $\tilde{b}$ and $\tilde{c}$ such that $[1,\tilde{b},\tilde{c},d\tilde{d}^{-1}] \in G$.

Now note that by Equation \eqref{eq:grouplaw} $$[1,\tilde{b},\tilde{c},d\tilde{d}^{-1}]^p=[1,0,c',\left(d\tilde{d}^{-1}\right)^p]$$ for a certain $c'$ and hence
$$[1,\tilde{b},\tilde{c},d\tilde{d}^{-1}]^{p^2}=[1,0,0,\left(d\tilde{d}^{-1}\right)^{p^2}].$$
Since $d\tilde{d}^{-1} \in \fqn$, this implies that $[1,0,0,d\tilde{d}^{-1}]=[1,0,0,(d\tilde{d}^{-1})^{q^{2n}}]=[1,\tilde{b},\tilde{c},d\tilde{d}^{-1}]^{q^{2n}}\in G.$ By Equation \eqref{eq:help2}, we conclude that $g \in G$ as desired. This shows that Equation \eqref{eq:intersection} holds, and thus completes the proof of the theorem.
\end{proof}

Combining Theorems \ref{thm:classifyH} and \ref{thm:GGKreducetoH}, we deduce the following characterization of all subgroups $G \leq B(Q_\infty)$.
\begin{theorem}\label{thm:classifyGGK}
For a subgroup $G \leq B(Q_\infty)$, let $G_0,G_2$ and $G_3$ be as above. Then up to conjugation we have the following:
\begin{enumerate}
\item[(i)] $G_0$ is a cyclic subgroup of $\mu$, the group of $m(q^2-1)$th roots of unity in $\fqn^*$.
\item[(ii)] $G_2 \subset \fqt$ is a vector space over $\mathbb{F}_p(G_0^m)$.
\item[(iii)] $G_3\subset \{c \in \fqt \mid c^q+c=0\}$ is a vector space over $\mathbb{F}_p(G_0^{q^n+1})$ containing $W$, where $W=\{b_1b_2^q-b_2b_1^q \, | \, b_1,b_2 \in G_2\}$ if $p$ is odd and $W=G_2^{q+1}=\{b^{q+1} \, | \, b \in G_2\}$ if $p=2$.
\end{enumerate}
Conversely, for any $G_0$, $G_2$, $G_3$ satisfying (i), (ii) and (iii) there exists a subgroup $G \leq B(Q_\infty)$ giving rise to this triple.
\end{theorem}

\begin{remark}
For $n \ge 5$ every automorphism of $\mathcal C_n$ fixes $Q_\infty$, so we have $B=B(Q_\infty)$ (see Equation \eqref{eq:index}). Hence in this case Theorem \ref{thm:classifyGGK} characterizes all subgroups of the automorphism group of the GGK function field.
\end{remark}

\section{Genus computation}

Let $G \leq B(Q_\infty)$ be a subgroup. We denote the fixed field of $G$ by $\mathcal C_n^G$. The aim of this section is to compute the genus of $\mathcal C_n^G$. As above, we let $G_0=\pi_d(G)$, $G_1=\pi_a(G)$, $U_{\pi(G)}=\ker \phi$, $G_2=\psi(U_{\pi(G)})$ and $G_3=\ker \psi$. Moreover, $g_0=\#G_0$, $g_1=\#G_1$, $g_w=\# U_{\pi(G)}=\#G_2\#G_3$ and $\#G=g_0g_w$.

We will adapt the approach in \cite{GSX} to calculate genera of Galois subfields of the GGK function field. A similar approach has been used in \cite{DO}. However since the genus formulas have not been worked out in full generality (it is assumed that the characteristic is odd in \cite{DO}), we do so below.

The fixed field of $B(Q_\infty)$ is $\fqn(t)$, where $t=z^{(q-1)(q^n+1)}$. Let us recall some details about the extension $\mathcal C_n/\fqn(t)$ and its ramification structure:
\begin{itemize}
  \item $(t=0)$ and $(t=\infty)$ are the only places ramified in the extension $\mathcal C_n/\fqn(t)$
  \item $(t=\infty)$ is totally ramified in the extension $\mathcal C_n/\fqn(t)$. The unique place of $\mathcal C_n$ lying over $(t=\infty)$ was denoted by $Q_{\infty}$, and its restriction to $\fqn(x,y)$ by $P_{\infty}$. A uniformizing element at $Q_{\infty}$ is given by $\tau=\frac{z^{q^n-3}}{x}$.
  \item $(t=0)$ is tamely ramified in  $\fqn(x,y)/\fqn(t)$ with ramification index $q^2-1$. The $q^3$ places of $\fqn(x,y)$ lying above $(t=0)$ are uniquely characterised by the value of $x$ and $y$ at those places, and hence will be denoted as
 $$\{P_{\alpha \beta}\;|\;\alpha,\beta \in \fqt \ , \alpha^q+\alpha=\beta^{q+1}\} \ .$$
Each place $P_{\alpha\beta}$ is totally ramified in the extension  $\mathcal C_n/\fqn(x,y)$, and we denote the unique place of $\mathcal C_n$ lying over it by $Q_{\alpha\beta}$.
 \end{itemize}

As shown in \cite{DO} for odd characteristic, the degree of the different divisor of the extension $\mathcal C_n/\mathcal C_n^G$ is given by
\begin{equation} \label{eq:different}
\deg\ {\rm Diff}(\mathcal C_n/\mathcal C_n^G)=\sum_{{\rm id}\neq g \in G} v_{Q_{\infty}}(g(\tau)-\tau)+\sum_{{\rm id}\neq g \in G} N(g) \ ,
\end{equation}
where  $v_{Q_{\infty}}$ is the valuation corresponding to the place $Q_{\infty}$, $\tau=z^{q^n-3}/x$ a uniformizing element at $Q_\infty$, and
$$N(g):=\# \{Q_{\alpha\beta}\;|\;g(Q_{\alpha\beta})=Q_{\alpha\beta}\}\ .$$
However, the proof given in \cite{DO} for Equation \eqref{eq:different} carries over directly to the even characteristic.
Since $Q_{\alpha\beta}|P_{\alpha\beta}$ is totally ramified in $\mathcal C_n/\fqn(x,y)$, we have the following conclusion.
\begin{lemma} \label{lem:orders}
$$N(g)=\# \{P_{\alpha\beta}\;|\;\pi(g)(P_{\alpha\beta})=P_{\alpha\beta}\}\ .$$
\end{lemma}

As a consequence of  Lemma~\ref{lem:orders} and \cite[Lemma 4.2]{GSX} we obtain the following description of $N(g)$ in terms of the order ${\rm ord}(\pi(g))$ of $\pi(g)$ in $A(P_\infty)$.
\begin{theorem} \label{thm:Ng}
Let $g\in B(Q_\infty)$. We have
$$N(g)=\left\{\begin{array}{ll}
0 & {\rm if}\ p\, |\, {\rm ord}(\pi(g)),\\
q^3 & {\rm if}\ \pi(g)={\rm id},\\
q & {\rm if}\ {\rm ord}(\pi(g))\, |\, (q+1),\\
1 & {\rm otherwise.}
\end{array}\right.
$$
\end{theorem}

In  \cite[Lemma 4.3]{GSX} the number of elements in a subgroup $H$ of $A(P_\infty)$ of various orders coprime to $p$ have been calculated. Since $\pi|_G:G\to A(P_\infty)$ is a group homomorphism with kernel of size ${\rm gcd}(g_0,m)$ (see Remark~\ref{rem:kernel}), we obtain the following:

\begin{lemma} \label{lem:noidea} Let $G$ be a subgroup of $B(Q_\infty)$, and  $a\in G_1$ an element of order of order $s>1$.
The number of elements $g\in G$ of the form $[a,\star,\star,\star]$ such that $\pi(g)$ has order $s$ if given by
\begin{itemize}
\item ${\rm gcd}(g_0,m)g_w$ if $s \nmid (q+1)$,
\item ${\rm gcd}(g_0,m)\#G_2$  if $s\mid (q+1)$.
\end{itemize}
\end{lemma}

Now we can calculate the expressions appearing in Equation~\eqref{eq:different} for the degree of the different in the extension $\mathcal C_n/\mathcal C_n^G$.
\begin{proposition}
Let $\delta_1={\rm gcd}(g_0,m)$, $\delta_2={\rm gcd}(g_0,q^n+1)$ and $\tau=z^{q^n-3}/x$. Then we have the following equalities.
\begin{enumerate}
\item $\displaystyle \sum_{{\rm id}\neq g \in G} v_{Q_{\infty}}(g(\tau)-\tau)=(m+g_0)g_w+(q^n+1-m)\#G_3-(q^n+2)$
\item $\displaystyle \sum_{{\rm id}\neq g \in G} N(g)=q(\delta_2-\delta_1)\#G_2+(g_0-\delta_2)g_w+q^3(\delta_1-1)$
\end{enumerate}
\end{proposition}
\begin{proof}
\begin{enumerate}
\item As was shown in $\cite{DM}$, for an element $g=[a,b,c,d]$, which is different than ${\rm id}$, we have
$$v_{Q_{\infty}}(g(\tau)-\tau)=\left\{\begin{array}{ll}
m+1 & {\rm if}\ d=1, b\neq 0,\\
q^n+2 & {\rm if}\ d=1, b=0,\\
1 & {\rm otherwise.}
\end{array}\right.$$
 The number of elements in $G\backslash\{{\rm id}\}$ with $d=1$ and $b\neq 0$ (respectively $b=0$) is given by $g_w-\#G_3$ (respectively $\#G_3-1$).  This leaves $\#G-g_w=g_0g_w-g_w$ elements for the third case.
\item Since $G_1$ is a cyclic group of order $g_1$, there are exactly ${\rm gcd}(g_1,q+1)-1$ elements in $G_1\backslash\{1\}$ of order dividing $q+1$. The remaining $(g_1-{\rm gcd}(g_1,q+1))$ elements in $G_1\backslash\{1\}$ have order not dividing $q+1$. The number of elements $g$ in $G\backslash\{{\rm id}\}$ such that $\pi(g)=[1,0,0]$ equals $\#\ker\pi|_G-1=\delta_1-1$. Then by using the facts that $g_1=g_0/\delta_1$ and ${\rm gcd}(g_1,q+1)=\delta_2/\delta_1$, we obtain the desired result from Theorem~\ref{thm:Ng} and Lemma~\ref{lem:noidea}.
\end{enumerate}
\end{proof}

Using the Riemann--Hurwitz genus formula for the extension $\mathcal C_n/\mathcal C_n^G$, we obtain the following expression for the genus of $\mathcal C_n^G$.
\begin{theorem}\label{thm:genus}
For $G \leq B(Q_\infty)$, let $\delta_1={\rm gcd}(g_0,m)$ and $\delta_2={\rm gcd}(g_0,q^n+1)$. Then we have
$$g(\mathcal C_n^G)=\frac{q^2(q^n+1)-q^3\delta_1-q(\delta_2-\delta_1)\#G_2+(\delta_2-m)g_w-(q^n+1-m)\#G_3}{2\#G}\ .$$
\end{theorem}

\begin{remark}
For $n=1$ we have $m=\delta_1=1$, $g_0=g_1$ and $\delta_2=\gcd(g_1,q+1)$, and we obtain
$$g(\mathcal H^G)=\frac{q-\#G_3}{2\#G}\left(q-(\delta_2-1)\#G_2\right)\ ,$$
as in \cite[Thm 4.4]{GSX}.
\end{remark}

Combining possible values for $g_1,\#G_2,\#G_3$ as given in \cite{bassa} with Theorem \ref{thm:genus} we can obtain many genera of maximal function fields. 
For several of these, no maximal function field of this genus was previously known to the best of our knowledge. We compared our results with the genera of maximal function fields given in \cite{AQ,bassa,DO,DO2,FG,GSX,auto1}. Some of the new genera for small values of $n$ and $q$ are as follows.
$$
{\renewcommand{\arraystretch}{1.5}
\begin{array}{l|l}
\mathbb{F}_{q^{2n}} & \makebox{new genera}\\
\hline
\mathbb{F}_{2^{12}} & 18 \\
\mathbb{F}_{2^{18}} & 37,45,82,99,189,207,244,406,840,1708\\
\mathbb{F}_{2^{20}} & 52,502,2552\\
\mathbb{F}_{3^{12}} & 16400, 17437, 52456\\
\mathbb{F}_{3^{14}} & 1365,2731,4369,8739\\
\mathbb{F}_{3^{18}} & 49,330,2065, 27280, 39388, 47775,54532, 54588, 78736, 78816, 95466,95550, \\
&109092,  118201, 157512, 190932, 236281, 236523, 354640, 472683, 708916, 709644, \\
& 1062856, 1418196, 1534612, 1860033, 2125740, 3069264, 3720066, 4605241,\\
& 9210603, 13817128, 27634620 \\
\mathbb{F}_{5^{6}} & 99, 285\\
\mathbb{F}_{5^{10}} & 24186, 37450, 64492\\
\mathbb{F}_{5^{12}} & 124804, 1874462, 2539056, 3124904, 7617168, 9374712\\
\end{array}
}$$

\section*{Acknowledgments}

Nurdag\"{u}l Anbar and Peter Beelen gratefully acknowledge the support from The Danish Council for Independent Research (Grant No.~DFF--4002-00367). Nurdag\"{u}l Anbar is also supported by H.C.~\O rsted COFUND Post-doc Fellowship from the project ``Algebraic curves with many rational points".  Alp Bassa is supported by the BAGEP Award of the Science Academy with funding supplied by Mehve{\c s} Demiren in memory of Selim Demiren.

\noindent
Nurdag\"ul Anbar\\
Technical University of Denmark,
Department of Applied Mathematics and Computer Science,
Matematiktorvet 303B, 2800 Kgs. Lyngby,
Denmark,
nurdagulanbar2@gmail.com

\vspace{1ex}
\noindent
Alp Bassa\\
Bo\u{g}azi\c{c}i University,
Faculty of Arts and Sciences,
Department of Mathematics,
34342 Bebek, \.{I}stanbul,
Turkey,
alp.bassa@boun.edu.tr

\vspace{1ex}
\noindent
Peter Beelen\\
Technical University of Denmark,
Department of Applied Mathematics and Computer Science,
Matematiktorvet 303B, 2800 Kgs. Lyngby,
Denmark,
pabe@dtu.dk

\end{document}